\documentclass{amsart}

\usepackage{amssymb,verbatim}
\numberwithin{equation}{section}

\theoremstyle{plain}
\newtheorem{theorem}{Theorem}[section]

\newtheorem{lemma}[theorem]{Lemma}
\newtheorem{proposition}[theorem]{Proposition}
\theoremstyle{remark}
\newtheorem{remark}[theorem]{Remark}
\theoremstyle{definition}
\newtheorem{definition}[theorem]{Definition}
\newtheorem{notation}[theorem]{Notation}

\makeatletter

\def\R{\mathbb{R}}
\def\C{\mathbb{C}}

\newcommand{\lie}[1]{\mathfrak{#1}}
\newcommand{\lier}[1]{\mathfrak{#1}_{0}}

\def\gK{\lie{g},K}
\def\brgK{(\gK)}

\def\Ug{U(\lie{g})}
\def\Zg{Z(\lie{g})}

\def\Hom{\mathrm{Hom}}

\def\Ker{\mathrm{Ker}}

\def\Pf{\mathrm{Pf}}
\def\PF{\mathbb{PF}}

\def\Ad{\mathrm{Ad}}
\def\ad{\mathrm{ad}}
\def\pr{\mathrm{pr}}

\def\sgn{\mathrm{sgn}}

\newcommand{\vect}[1]{\boldsymbol{#1}}
\def\I{\sqrt{-1}}

\allowdisplaybreaks[3]

\title[Central element of $U(\lie{so}_{n})$]
{A construction of generators of $Z(\lie{so}_{n})$}
\author{Kenji Taniguchi}
\address{
Department of Physics and Mathematics, 
Aoyama Gakuin University, 
5-10-1, Fuchinobe, Chuo-ku, Sagamihara, Kanagawa 252-5258, Japan. }
\email{taniken@gem.aoyama.ac.jp}

\begin{document}

\begin{abstract}
We construct generators of the center of the universal enveloping
algebra of the complex orthogonal Lie algebra realized as the
alternative matrices of size $n$. 
These elements are constructed in accordance with the Iwasawa
decomposition of the real rank one indefinite orthogonal Lie algebra. 
We also discuss the Iwasawa decomposition of the Pfaffian. 
\end{abstract}

\maketitle

\section{Introduction and Main results}
\label{section:introduction}

Let $\lie{g}=\lie{so}_{n}$ be the complex orthogonal Lie algebra
realized as the alternative matrices of size $n$. 
Denote by $\Ug$ the universal enveloping algebra of $\lie{g}$ and by
$\Zg$ the center of $\Ug$. 
It is well known for experts that a set of generators of $\Zg$ is
given with determinant and Pfaffian (\cite{HU}, \cite{MN},
\cite{IU}). 

The author is now interested in the structure of the space of
Whittaker functions on real reductive Lie groups. 
In order to determine the composition series of the standard Whittaker
$(\gK)$-modules (\cite{T2}), he needed to write the action of central
elements on the space of Whittaker functions. 
For the case of indefinite unitary group $U(n-1,1)$, 
he succeeded in this task by using the determinant type generators of
$Z(\lie{gl}_{n})$. 
He also tried the case of indefinite orthogonal group $SO_{0}(n-1,1)$, 
but, in his narrow idea, it seems that it is very difficult to write 
the differential equations characterizing Whittaker functions on
$SO_{0}(n-1,1)$ by using the above determinant type generators of
$Z(\lie{so}_{n})$. 
Under such backgrounds, he tried to write the action of
$Z(\lie{so}_{n})$ in a different way. 
As a result, a new construction of the generators of $Z(\lie{so}_{n})$
is obtained. 
This is the main object of this paper.

%%%%%%%%%%%%%%%

In order to express the main result, 
we introduce some notation. 
In general, for a real Lie group $L$, the Lie algebra of it is denoted
by $\lier{l}$ and its complexification by $\lie{l}$. 
This notation will be applied to groups denoted by other Roman letters
in the same way without comment. 
The Kronecker delta is denoted by $\delta_{i,j}$. 
Let $E_{i,j}:=(\delta_{i,k}\, \delta_{j,l})_{k,l=1}^{n}$ be the matrix
units and define $A_{j,i}=E_{j,i}-E_{i,j}$. 
These are the standard generators of the space of alternative matrices. 
The diagonal $n \times n$ matrix 
$\sum_{i=1}^{n-1} E_{i,i} - E_{n,n}$ is denoted by $I_{n-1,1}$. 
The field of real (resp. complex) numbers is denoted by $\R$
(resp. $\C$). 
For a complex matrix $Z = (z_{i,j})_{i,j}$, 
define $\overline{Z} = (\overline{z_{i,j}})_{i,j}$, where
$\overline{z}$ is the complex conjugate of a complex number $z$.

We realize the group $SO(n,\C)$ as the subgroup of $SL(n,\C)$
consisting of those elements which satisfy ${}^{t}g = g^{-1}$. 
Its Lie algebra $\lie{g} = \lie{so}_{n}$ is spanned by 
$A_{j,i}$, $1 \leq i < j \leq n$. 
The Lie group $G=SO_{0}(n-1,1)$ is the identity component of the real
form of $SO(n,\C)$ defined by the complex conjugation 
$g \mapsto I_{n-1,1} \overline{g} I_{n-1,1}$. 
Let $\theta(g)=I_{n-1,1} g I_{n-1,1}$ be a Cartan involution on 
$SO_{0}(n-1,1)$. 
Denote by $K$ the maximal compact subgroup  of $G$ consisting of the
fixed points of $\theta$. 
Let $\lier{g} = \lier{k} \oplus \lier{p}$ be the
corresponding Cartan decomposition of $\lier{g} = \lie{so}(n-1,1)$. 
More explicitly, 
\begin{align*}
& K \simeq SO(n-1), 
& &
\\
& \lier{k} 
= \R\mbox{-span}\{A_{j,i} \,|\, 1 \leq i < j \leq n-1\}, 
&
& \lier{p} 
= \R\mbox{-span}
(\{\I A_{n,i} \,|\, 1 \leq i \leq n-1\}. 
\end{align*}
As a maximal abelian subspace $\lier{a}$ of $\lier{p}$, we choose 
\[
\lier{a} = \R H, 
\qquad 
H := \I A_{n,n-1}. 
\]
The subgroup $\exp \lier{a}$ is denoted by $A$. 
Define a basis $\alpha$ of the complex dual space $\lie{a}^{\ast}$ by
$\alpha(H) = 1$. 
Then $\Sigma^{+}=\{\alpha\}$ is a positive system of the root
system $\Sigma(\lier{g},\lier{a})$. 
Dente by $\lier{n}$ the nilpotent subalgebra corresponding to
$\Sigma^{+}$. 
We choose 
\begin{align*}
X_{i} := A_{n-1,i} + \I A_{n,i}, 
\qquad 
i=1,\dots,n-2
\end{align*}
as a basis of $\lier{n}$. 
Define $N := \exp \lier{n}$. 
Then we get Iwasawa decompositions $G=NAK$ and 
$\lier{g} = \lier{n} \oplus \lier{a} \oplus \lier{k}$. 
As a consequence of Poincar\'e-Birkhoff-Witt theorem, 
$\Ug$ is isomorphic to 
$U(\lie{n}) \otimes U(\lie{a}) \otimes U(\lie{k})$. 

As usual, denote by $M$ the centralizer of $A$ in $K$. 
This group is isomorphic to $SO(n-2)$. 
Define a Cartan subalgebra $\lie{t}_{\lie{m}}$ of $\lie{m}$ by 
\[
\lie{t}_{\lie{m}}
:= \sum_{i=1}^{\lfloor (n-2)/2 \rfloor} 
\C T_{i}, 
\qquad 
T_{i} := \I A_{n-2i,n-1-2i}. 
\]
We set $\lie{t} = \lie{t}_{\lie{m}}$ if $n$ is even, and 
$\lie{t} = \lie{t}_{\lie{m}} + \C T_{\lfloor (n-1)/2 \rfloor}$ if $n$
is odd. 
Here $T_{\lfloor (n-1)/2 \rfloor} := \I A_{n-1,1}$. 
This $\lie{t}$ is a Cartan subalgebra of $\lie{k}$. 
Define a basis $\{e_{1}, \dots, e_{\lfloor (n-1)/2 \rfloor}\}$ 
of $\lie{t}^{\ast}$ by $e_{i} (T_{j}) = \delta_{i,j}$. 
We regard $\{e_{1}, \dots, e_{\lfloor (n-2)/2 \rfloor}\}$ as a basis
of $(\lie{t}_{\lie{m}})^{\ast}$. 

Choose a Borel subalgebra 
$\lie{b}_{\lie{m}} = \lie{t}_{\lie{m}} \oplus \lie{u}$ of $\lie{m}$. 
Set $\lie{h} := \lie{t}_{\lie{m}} \oplus \lie{a}$ and 
$\lie{N} := \lie{n} \oplus \lie{u}$. 
The nilpotent subalgebras opposite to $\lie{n}$, $\lie{u}$ and
$\lie{N}$ are denoted by $\overline{\lie{n}}$, 
$\overline{\lie{u}}$ and $\overline{\lie{N}}$, respectively. 
Denote by $\gamma$ the Harish-Chandra map defined by the projection 
$U(\lie{g}) 
\simeq 
U(\lie{h}) \oplus (\lie{N} U(\lie{g}) + U(\lie{g}) \overline{\lie{N}})
\rightarrow U(\lie{h})$ composed by rho shift.

\begin{theorem}\label{theorem:main-1}
Suppose $\lie{g} = \lie{so}_{n}$. 
Let 
\[
\Omega_{n-2}
=
\sum_{1 \leq i < j \leq n-2} (A_{j,i})^{2}
\]
be a multiple of the Casimir element of $\lie{so}_{n-2}$. 
For a parameter $u \in \C$, 
define elements $C_{n}(u) \in U(\lie{g})$ 
inductively by the following formulas: 
\begin{align}
C_{0}(u) =& C_{1}(u) = 1, 
\notag\\
C_{n}(u) =& 
-\left\{\left(H-\frac{n-2}{2}\right)^{2} - u^{2} 
+ \sum_{i=1}^{n-2} X_{i}^{2}
\right\} 
C_{n-2}(u) 
\label{eq:main formula}\\
&+ 
\sum_{i=1}^{n-2} X_{i} \left(H-\frac{n-5}{2}\right) 
[A_{n-1,i}, C_{n-2}(u)] 
+ 2 \sum_{i=1}^{n-2} X_{i}\, C_{n-2}(u) A_{n-1,i} 
\notag\\
&-\frac{1}{2} 
\sum_{i=1}^{n-2} X_{i}\, 
[\Omega_{n-2}, [A_{n-1,i}, C_{n-2}(u)]] 
\notag\\
&-\frac{1}{2} 
\sum_{i,j=1}^{n-2} 
X_{i} X_{j}\, 
[A_{n-1,i}, [A_{n-1,j}, C_{n-2}(u)]] 
%\notag\\
\quad
\mbox{for} %& 
\quad n=2,3,\dots
\notag
\end{align}
Then $C_{n}(u)$ is an element of $Z(\lie{g})$ for any $u \in \C$. 

Moreover, the image $\gamma(C_{n}(u))$ of the Harish-Chandra map
$\gamma$ is 
\begin{equation}\label{eq:image of HC map}
\gamma(C_{n}(u)) = 
(u^{2}-H^{2}) (u^{2}-T_{1}^{2}) \cdots 
(u^{2}-T_{\lfloor (n-2)/2 \rfloor}^{2}). 
\end{equation}
\end{theorem}

This paper is organized as follows. 
In \S\ref{section:shift operators}, we explain the $K$-type shift
operators and write it explicitly. 
\S\ref{section:proof of the main theorem} is the main part of this
paper, in which Theorem~\ref{theorem:main-1} is proved. 
For the proof, we relate the element $C_{n}(u)$ and some composition
of $K$-type shift operators. 
This relationship is explained in Lemma~\ref{lemma:shift-central}. 
The proof of this lemma is %elementary but messy. 
%It is 
done in \S\ref{section:proof of lemma}. 
For completeness, we discuss the Iwasawa decomposition of the
Pfaffian in \S\ref{section:Pfaffian} and we relate it to a $K$-type
shift operator. 

%%%%%%%%%%%%%%%%

\section{Shift operators}\label{section:shift operators} 

In order to show that $C_{n}(u)$ is invariant under the adjoint action
of $K$, 
we use the $K$-type shift operators on the space of smooth 
functions on $G$. 

Let us review the definition of shift operators briefly. 
Denote by $\lie{g} = \lie{k} \oplus \lie{p}$ the complexified Cartan
decomposition. 
For a finite dimensional representation $(\tau, V)$ of $K$, define 
\begin{align*}
C_{\tau}^{\infty}(K\backslash G)
:=
\{f : G \overset{C^{\infty}}{\longrightarrow} V
\,|\, f(kg) = \tau(k) f(g), 
\enskip k \in K, g \in G\}. 
\end{align*}
This space is isomorphic to the intertwining space 
$\Hom_{K}(V^{\ast}, C^{\infty}(G)_{K\mbox{-}\mathrm{finite}})$, 
where $V^{\ast}$ is the
contragredient representation of $(\tau,V)$. 

Choose an orthonormal basis $\{W_{i}\}$ of $\lier{p}$ with respect to
an invariant bilinear form $\langle{\enskip},{\enskip}\rangle$ 
on $\lie{g}$ which is negative
(resp. positive) definite on $\lier{k}$ (resp. $\lier{p}$). 
Define a differential-difference operator $\nabla$ by 
\[
\nabla \phi_{\tau} := \sum_{i} L(W_{i}) \phi_{\tau} \otimes W_{i}, 
\qquad
\phi_{\tau} \in C_{\tau}^{\infty}(K\backslash G). 
\]
Here, $L(\ast)$ is the left regular representation. 
It is easy to see that $\nabla$ does not depend on the choice of an
orthonormal basis $\{W_{i}\}$ of $\lier{p}$. 
As a consequence, the image of $\nabla$ is an element of 
$C_{\tau \otimes \Ad}^{\infty}(K \backslash G)$; 
\begin{align}
& \nabla \phi_{\tau}(kg) = 
(\tau \otimes \Ad)(k)\, \nabla \phi_{\tau}(g), 
\quad k \in K, g \in G, 
\label{eq:K-equivariant}
\end{align}
Here ``$\Ad$'' is the adjoint representation of $K$ on $\lie{p}$.

Let $\lambda \in \lie{t}^{\ast}$ be a dominant integral weight of $K$ 
and let $(\tau_{\lambda}, V_{\lambda})$ be the irreducible
representation of $K$ with highest weight $\lambda$. 
For notational convenience, 
set 
$e_{-\ell} = - e_{\ell}$ for $\ell = 1,2,\dots$ and $e_{0} = 0$. 
In the case when $G = SO_{0}(n-1,1)$ and $\lambda$ is sufficiently
regular, 
the irreducible decomposition of $V_{\lambda} \otimes \lie{p}$ is 
\begin{equation}\label{eq:irred decomp} 
V_{\lambda} \otimes \lie{p}
\simeq 
\bigoplus_{1 \leq |\ell| \leq \lfloor (n-1)/2 \rfloor} 
V_{\lambda+e_{\ell}}
\left(
\oplus V_{\lambda+e_{0}}
\mbox{ if $n$ is even}\right). 
\end{equation}
The projection operator from $V_{\lambda} \otimes \lie{p}$ to 
$V_{\lambda+e_{\ell}}$ is denoted by $\pr_{\ell}$. 
Define a $K$-type shift operator $P_{\ell}$ by 
\[
P_{\ell} = \pr_{\ell} \circ \nabla 
: 
C_{\tau_{\lambda}}^{\infty}(K \backslash G) 
\to 
C_{\tau_{\lambda+e_{\ell}}}^{\infty}(K \backslash G). 
\]

The basis 
$\{\I A_{n,i} \,|\, 1 \leq i \leq n-1\}$ 
of $\lier{p}$ is orthonormal with respect to an appropriately
normalized invariant bilinear form. 
Therefore, 
the operator $\nabla$ is 
\[
\nabla \phi_{\tau}(g) 
= 
\sum_{i=1}^{n-1} 
L(\I A_{n,i}) \, \phi_{\tau}(g) \otimes \I A_{n,i}. 
\]
For an irreducible representation $\tau$ of $K \simeq SO(n-1)$, 
the shift operators are explicitly calculated in \cite{T1}. 
To state the results, we introduce the Gelfand-Tsetlin basis of
irreducible representations of $SO(n-1)$. 

\begin{definition}\label{definition:GT}
Let 
$\lambda = (\lambda_{1}, \dots, \lambda_{\lfloor (n-1)/2 \rfloor})$ 
be a dominant integral weight of $SO(n-1)$. 
A {\it ($\lambda$-)Gelfand-Tsetlin pattern} is a set of vectors 
$Q=(\vect{q}_{1}, \dots, \vect{q}_{n-2})$ such that 
\begin{enumerate}
\item
$\vect{q}_i=(q_{i,1}, q_{i,2}, \dots, q_{i,\lfloor(i+1)/2\rfloor})$. 
\item
The numbers $q_{i,j}$ are all integers. 
\item
$q_{2i+1,j} \geq q_{2i,j} \geq q_{2i+1,j+1}$, 
for any $j=1, \dots, i-1$.  
\item
$q_{2i+1,i} \geq q_{2i,i} \geq |q_{2i+1,i+1}|$.  
\item
$q_{2i,j} \geq q_{2i-1,j} \geq q_{2i,j+1}$, 
for any $j=1, \dots, i-1$. 
\item
$q_{2i,i} \geq q_{2i-1,i} \geq -q_{2i,i}$.  
\item
$q_{n-2,j}=\lambda_{j}$. 
\end{enumerate}
The set of all $\lambda$-Gelfand-Tsetlin patterns is denoted by
$GT(\lambda)$. 
\end{definition}
\begin{notation}
For any set or number $\ast$ depending on $Q \in GT(\lambda)$, 
we denote it by $\ast(Q)$, if we need to specify $Q$. 
For example, $q_{i,j}(Q)$ is the $q_{i,j}$ part of $Q \in GT(\lambda)$. 
\end{notation}
\begin{theorem}[\cite{GT}]\label{theorem:GT} 
For a dominant integral weight $\lambda$ of $SO(n-1)$, 
the set $GT(\lambda)$ of Gelfand-Tsetlin patterns is identified with a
basis of $(\tau_{\lambda}, V_{\lambda})$. 

The action of the elements 
%$A_{p,q} \in 
in $\lie{so}(n-1)$ is expressed as
follows. 
For $j > 0$, let 
\begin{align*}
& 
l_{2i-1,j} := q_{2i-1,j} + i - j, 
& 
& 
l_{2i-1,-j} := - l_{2i-1,j}, 
\\
& 
l_{2i,j} := q_{2i,j} + i + 1 - j, 
& 
& 
l_{2i,-j} := - l_{2i,j} + 1,  
\end{align*}
and let $l_{2i,0} = 0$. 
Define $a_{p,q}(Q)$ by 
\begin{align*} 
a_{2i-1,j}(Q) 
&= 
\sgn j \, 
\sqrt{-
\frac{\prod_{1 \leq |k| \leq i-1}(l_{2i-1,j} + l_{2i-2, k}) 
\prod_{1 \leq |k| \leq i} (l_{2i-1,j} + l_{2i, k})}
{4 \prod_{\genfrac{}{}{0pt}{}{1 \leq |k| \leq i,}{k \not= \pm j}} 
(l_{2i-1,j} + l_{2i-1,k}) (l_{2i-1,j} + l_{2i-1,k} + 1)}
},
\intertext{for $j = \pm 1, \dots, \pm i$, and}  
a_{2i,j}(Q) 
&= 
\epsilon_{2i,j}(Q) 
\sqrt{-
\frac{\prod_{1 \leq |k| \leq i}(l_{2i,j} + l_{2i-1, k}) 
\prod_{1 \leq |k| \leq i+1} (l_{2i,j} + l_{2i+1, k})}
{(4 l_{2i,j}^{2} - 1) 
\prod_{\genfrac{}{}{0pt}{}{0 \leq |k| \leq i}{k \not= \pm j}} 
(l_{2i,j} + l_{2i,k}) (l_{2i,j} - l_{2i,k})}
},
\end{align*}
for $j = 0, \pm 1, \dots, \pm i$, 
where $\epsilon_{2i,j}(Q)$ is $\sgn j$ if $j \not= 0$, 
and $\sgn(q_{2i-1,i}\, q_{2i+1,i+1})$ if $j = 0$. 
Let $\sigma_{a,b}$ be the shift operator, sending $\vect{q}_a$ to 
$\vect{q}_{a} + (0, \dots, \overset{|b|}{\sgn(b)}, 0 ,\dots, 0)$. 

Under the above notation, the action of the Lie algebra is expressed
as 
\begin{align*}
\tau_\lambda(A_{2i+1,2i})Q
&=
\sum_{1 \leq |j| \leq i} a_{2i-1,j}(Q)\, \sigma_{2i-1,j}Q, 
\\
\tau_\lambda(A_{2i+2,2i+1})Q
&=
\sum_{0 \leq |j| \leq i} a_{2i,j}(Q)\, \sigma_{2i,j}Q. 
\end{align*}
\end{theorem}

\begin{remark}\label{remark:GT restriction}
This basis is compatible with the restriction to
smaller orthogonal groups. 
More precisely, the restriction of $\tau_{\lambda}$ to $SO(n-2)$ is
multiplicity free, 
and the vector $Q = (\vect{q}_{1}, \dots, \vect{q}_{n-2})$ is
contained in the irreducible representation of $SO(n-2)$ 
whose highest weight is $\vect{q}_{n-2}$.
\end{remark}

In order to write the projection operator $\pr_{\ell}$
explicitly, 
we embed $V_{\lambda}$ and $V_{\lambda+e_{\ell}}$ into an
appropriately chosen irreducible representation of $SO(n)$. 
For example, when we consider the projection $\pr_{1}$, 
we embed $V_{\lambda}$ and $V_{\lambda+e_{1}}$ into the irreducible
representation of $SO(n)$ whose highest weight is 
$\widetilde{\lambda} = (\lambda_{1}+1, \lambda_{2}, \dots)$. 
If we do so, then ``$a_{n-2,\ell}(Q) \sigma_{n-2,\ell} Q$'' 
in the following (for example in \eqref{eq:tensor A_{n,n-1}}) 
makes sense.

Just in the way as the proof of \cite[Proposition~4.3]{Kr}, 
we get the following formulas. 
\begin{lemma}\label{lemma:projection-1}
For $Q \in GT(\lambda)$ and 
$\ell = 0, \pm 1, \dots, \pm \lfloor (n-1)/2 \rfloor$, 
\begin{align}
& 
\pr_{\ell}(Q \otimes \I A_{n,n-1}) 
= 
a_{n-2,\ell}(Q) \sigma_{n-2,\ell} Q. 
\label{eq:tensor A_{n,n-1}}
\end{align}
\end{lemma}
\begin{remark}\label{remark:varpi}
This lemma says that, 
if we embed $V_{\lambda}$ and $V_{\lambda+e_{\ell}}$ into an
irreducible representation $V_{\widetilde{\lambda}}$ of $SO(n)$, 
then we may identify $\pr_{\ell}(Q \otimes \I A_{n,n-1})$ 
with the 
$V_{\lambda+e_{\ell}} \subset V_{\widetilde{\lambda}}|_{SO(n-1)}$
component of $\tau_{\widetilde{\lambda}}(A_{n,n-1}) Q \in 
V_{\widetilde{\lambda}}$. 
\end{remark}

Let us write the operator $P_{\ell}$ explicitly. 
The action of $\lie{k}$ on $\phi_{\tau_{\lambda}}$ is given
by  
\begin{equation}\label{eq:action of Lie algebra elements}
L(W) \phi_{\tau_{\lambda}}(a) 
= -\tau_{\lambda}(W) \phi_{\tau_{\lambda}}(a) \quad 
\mbox{for $W \in \lie{k}$}. 
\end{equation}
Let $\varpi_{\ell}$ be operators from $GT(\lambda)$ to
$GT(\lambda+e_{\ell})$ defined by 
\begin{equation}\label{eq:varpi}
\varpi_{\ell} Q := a_{n-2,\ell}(Q) \sigma_{n-2,\ell} Q,
\qquad \ell = 0, \pm 1, \dots, \lfloor (n-1)/2 \rfloor. 
\end{equation}
Here, 
$\varpi_{0}$ is defined only when $n$ is even. 
Then \eqref{eq:tensor A_{n,n-1}} is 
$\pr_{\ell}(Q \otimes \I A_{n,n-1})
= 
\varpi_{\ell}Q$, 
and 
\begin{align*}
\pr_{\ell}&(Q \otimes \I A_{n,i}) 
\\
&= 
\pr_{\ell}(\tau_{\lambda}(A_{n-1,i}) Q \otimes \I A_{n,n-1})
%\\
%& \qquad 
-
\pr_{\ell}\{(\tau_{\lambda} \otimes \ad)(A_{n-1,i})
(Q \otimes \I A_{n,n-1})\}
\\
&=
\varpi_{\ell} \tau_{\lambda}(A_{n-1,i}) Q 
-\tau_{\lambda+e_{\ell}}(A_{n-1,i}) \varpi_{\ell} Q. 
\end{align*}
For simplicity, we omit the symbols $\tau_{\lambda}$ and
$\tau_{\lambda+e_{\ell}}$ hereafter. 
For example, we write 
$[\varpi_{\ell}, A_{n-1,i}]$ instead of  
$\varpi_{\ell} \tau_{\lambda}(A_{n-1,i}) 
-\tau_{\lambda+e_{\ell}}(A_{n-1,i}) \varpi_{\ell}$, 
so the projection above is 
\[
\pr_{\ell}(Q \otimes \I A_{n,i}) 
= 
[\varpi_{\ell}, A_{n-1,i}] Q. 
\]

In order to express $\phi_{\tau_{\lambda}}(g) 
\in C_{\tau_{\lambda}}^{\infty}(K \backslash G)$ explicitly, 
we use the Gelfand-Tsetlin basis. 
The coefficient function of $Q$ is denoted by $c(Q; g)$. 
Namely, we write 
\[
\phi_{\tau_{\lambda}}(g) = \sum_{Q \in GT(\lambda)} c(Q; g)\, Q.
\]
\begin{lemma}[\cite{T1}]\label{lemma:K-type shift}
For $\ell = 0, \pm 1, \dots, \pm \lfloor (n-1)/2 \rfloor$, 
the following formulas hold: 
\begin{align}
P_{\ell} \phi_{\tau_{\lambda}}(g)
= 
\sum_{Q \in GT(\lambda)} 
\bigg\{&
(L(H)+l_{n-2,\ell}-\lfloor \frac{n-2}{2} \rfloor)\, 
c(Q; g)\, \varpi_{\ell}Q 
\label{eq:P_k}\\
&
+\sum_{i=1}^{n-2} L(X_{i}) c(Q; g) [\varpi_{\ell}, A_{n-1,i}] Q
\bigg\}.  
\notag
\end{align}
\begin{align}
P_{-\ell} P_{\ell} \phi_{\tau_{\lambda}}(g) &
\notag
\\
= 
\sum_{Q \in GT(\lambda)} 
\bigg\{&
(L(H) - l_{n-2,\ell}-\lfloor \frac{n-1}{2} \rfloor)
(L(H) + l_{n-2,\ell}-\lfloor \frac{n-2}{2} \rfloor)\, c(Q; g)\, 
\varpi_{-\ell} \varpi_{\ell}Q 
\notag
\\
& 
+\sum_{i=1}^{n-2} 
L(X_{i}) (L(H) - l_{n-2,\ell}-\lfloor\frac{n-3}{2}\rfloor) 
c(Q; g) 
\varpi_{-\ell} [\varpi_{\ell}, A_{n-1,i}] Q
\label{eq:P_{-k} P_k}
\\
& 
+\sum_{i=1}^{n-2} 
L(X_{i}) (L(H)+l_{n-2,\ell}-\lfloor \frac{n-2}{2}\rfloor) 
c(Q; g) 
[\varpi_{-\ell}, A_{n-1,i}] \varpi_{\ell} Q
\notag
\\
& 
+ \sum_{i,j=1}^{n-2} 
L(X_{i}) L(X_{j}) c(Q; g) 
[\varpi_{-\ell}, A_{n-1,i}] [\varpi_{\ell}, A_{n-1,j}] Q
\bigg\}. 
\notag
\end{align}
\end{lemma}
\begin{proof}
When $n$ is odd, \eqref{eq:P_k} is obtained in 
\cite[Proposition~5.1.4]{T1}. 
The calculation there is also valid if $n$ is even. 
\eqref{eq:P_{-k} P_k} can be obtained by composing two operators
$P_{-\ell}$ and $P_{\ell}$. 
For the proof of \eqref{eq:P_{-k} P_k}, 
we use the following identity: If $\ell\not=0$, 
\begin{align*}
l_{n-2,-\ell}&(\sigma_{n-2,\ell}Q) - \lfloor \frac{n-2}{2}\rfloor
=
l_{n-2,-\ell}(Q) -1 - \lfloor \frac{n-2}{2}\rfloor
\\
&=\begin{cases}
- l_{n-2,\ell}(Q) - \frac{n-2}{2} \quad (n \mbox{ is even})
\\
- l_{n-2,\ell}(Q) - 1 - \frac{n-3}{2} \quad (n \mbox{ is odd}) 
\end{cases}
= 
- l_{n-2,\ell}(Q) - \lfloor \frac{n-1}{2}\rfloor.
\end{align*}
The conclusion of this equality is valid when $n$ is even and $\ell=0$. 
\end{proof}

%%%%%%%%%%%%%%%%%%%%%%%%%%%%%%%%%%%%%%%%%

\section{Proof of the main theorem}
\label{section:proof of the main theorem}
In this section, we prove Theorem~\ref{theorem:main-1}. 
For the proof, we introduce notation. 
\begin{definition}\label{definition:u_l}
Define 
\[
u_{\ell} = 
\begin{cases}
\lambda_{\ell}+n/2-\ell, \quad \mbox{if $\ell > 0$}, 
\\
1-(\lambda_{|\ell|}+n/2-|\ell|), \quad \mbox{if $\ell < 0$}, 
\\
0, \quad \mbox{when $n$ is even and $\ell=0$.}
\end{cases}
\] 
In other words, 
$u_{\ell} = l_{n-2,\ell} + 1/2$ when $n$ is odd and 
$u_{\ell} = l_{n-2,\ell}$ when $n$ is even. 
\end{definition}

Next lemma is a key to show our main theorem. 
\begin{lemma}\label{lemma:shift-central}
Assume that $C_{n-2}(u) \in Z(\lie{so}_{n-2})$ and 
that it satisfies \eqref{eq:image of HC map}. 
For $\ell = 0, \pm 1, \dots, \pm \lfloor(n-1)/2\rfloor$ and 
for every irreducible representation 
$(\tau_{\lambda}, V_{\lambda})$ of $K \simeq SO(n-1)$, there exists 
a non-zero constant $d_{\lambda,\ell}$ determined by $\ell$ and the
highest weight $\lambda$ such that 
\begin{equation}\label{eq:shift-central}
P_{-\ell} P_{\ell} \phi_{\tau_{\lambda}}(g) 
= 
d_{\lambda,\ell}\, 
L(C_{n}(u_{\ell})) \phi_{\tau_{\lambda}}(g), 
\qquad 
\phi_{\tau_{\lambda}} 
\in 
C_{\tau_{\lambda}}^{\infty}(K \backslash G). 
\end{equation}
\end{lemma}

This lemma is proved by direct calculation. 
Since it is elementary but messy, 
we prove it in the next section. 
Here we complete the proof of Theorem~\ref{theorem:main-1}. 

\begin{proposition}\label{proposition:K-invariance of C_n}
Assume that $C_{n-2}(u) \in Z(\lie{so}_{n-2})$ and 
that it satisfies \eqref{eq:image of HC map}. 
For any $k \in K \simeq SO(n-1)$, 
$\Ad(k) C_{n}(u) = C_{n}(u)$. 
\end{proposition}
\begin{proof}
Let $X(\mu, \nu)$, $\mu \in \widehat{M}$, 
$\nu \in \lie{a}^{\ast}$ be the Harish-Chandra module of the principal
series representation induced from $\mu \boxtimes e^{\nu+\rho}$. 
If $G=SO_{0}(n-1,1)$, 
then 
$X(\mu, \nu)$ is $K$-multiplicity free because of the Frobenius
reciprocity 
$\Hom_{K}(V_{\lambda}^{\ast}, X(\mu,\nu)) 
\simeq 
\Hom_{M}(V_{\lambda}^{\ast}, \mu)$ 
and the multiplicity freeness of $V_{\lambda}|_{SO(n-2)}$. 

Suppose that a function 
$\phi_{\tau_{\lambda}} 
\in 
C_{\tau_{\lambda}}^{\infty}(K \backslash G)$ 
corresponds to an intertwining operator in 
$\Hom_{K}(V_{\tau_{\lambda}}^{\ast}, X(\mu,\nu))$. 
Since $X(\mu,\nu)$ is $K$-multiplicity free, 
the function $P_{-\ell} P_{\ell} \phi_{\tau_{\lambda}}$ is a constant
multiple of $\phi_{\tau_{\lambda}}$. 
It follows that 
\[
L(k) P_{-\ell} P_{\ell} L(k^{-1}) \phi_{\tau_{\lambda}} 
= P_{-\ell} P_{\ell} \phi_{\tau_{\lambda}}. 
\]
By Lemma~\ref{lemma:shift-central}, 
$P_{-\ell} P_{\ell} \phi_{\tau_{\lambda}} 
= 
d_{\lambda,\ell}\, L(C_{n}(u_{\ell})) \phi_{\tau_{\lambda}}$. 
Therefore, 
\begin{align*}
L(\Ad(k)\, C_{n}(u_{\ell})) \, \phi_{\tau_{\lambda}}
&= 
L(k)\, L(C_{n}(u_{\ell}))\, L(k^{-1})\, \phi_{\tau_{\lambda}} 
\\
&= 
(d_{\lambda,\ell})^{-1} L(k)\, P_{-\ell} P_{\ell} \, L(k^{-1})\, 
\phi_{\tau_{\lambda}} 
\\
&= 
(d_{\lambda,\ell})^{-1} P_{-\ell} P_{\ell} \,
\phi_{\tau_{\lambda}} 
\\
&= 
L(C_{n}(u_{\ell}))\; \phi_{\tau_{\lambda}}. 
\end{align*}
By the definition of $l_{n-2,\ell}$ and $u_{\ell}$, 
the $u_{\ell}$'s satisfy 
$u_{1} > u_{2} > \dots > u_{\lfloor (n-1)/2 \rfloor} > 0$ 
if $n$ is odd and 
$u_{1} > u_{2} > \dots > u_{\lfloor (n-1)/2 \rfloor} > u_{0} = 0$ if
$n$ is even. 
It follows that $L(\Ad(k) C_{n}(u)) \phi_{\tau_{\lambda}}$ and
$L(C_{n}(u)) \phi_{\tau_{\lambda}}$ are identical for 
$\lfloor n/2 \rfloor$ points 
$u^{2} = u_{1}^{2}, \dots, u_{\lfloor n/2 \rfloor}^{2}$. 
By the definition \eqref{eq:main formula} of $C_{n}(u)$, 
it is a {\it monic} polynomial in $u^{2}$ 
of degree $\lfloor n/2 \rfloor$. 
Therefore 
$L(\Ad(k) C_{n}(u)) \phi_{\tau_{\lambda}}
= L(C_{n}(u)) \phi_{\tau_{\lambda}}$ 
for any $u \in \C$ and for every $K$-type 
$(\tau_{\lambda}^{\ast}, V_{\tau_{\lambda}}^{\ast})$ in
$X(\mu,\nu)$. 
It follows that $\Ad(k) C_{n}(u) - C_{n}(u)$ annihilates every
principal series. 
By the subrepresentation theorem, it annihilates every irreducible
Harish-Chandra module. 
According to the Plancherel formula for $G$, 
the element $\Ad(k) C_{n}(u) - C_{n}(u)$ acts trivially on the space 
$C_{c}^{\infty}(G) (\subset L^{2}(G))$ of smooth functions of compact
support, 
so it is the zero element in $\Ug$. 
\end{proof}

\noindent
{\it Proof of Theorem~\ref{theorem:main-1}.} 
We shall show $C_{n}(u) \in Z(\lie{so}_{n})$ for any $u \in \C$ 
and it
satisfies \eqref{eq:image of HC map} by induction on $n$. 
If $n=0,1$, then this is trivial since $C_{0}(u) = C_{1}(u) = 1$ by
definition. 

Assume that $C_{n-2}(u) \in Z(\lie{so}_{n-2})$ and 
that it satisfies \eqref{eq:image of HC map}. 
Consider the (generalized) Harish-Chandra maps 
\begin{align}
& 
\gamma_{\lie{n}}: 
U(\lie{g}) 
\simeq U(\lie{m} \oplus \lie{a}) \oplus (\lie{n} U(\lie{g}) 
+ U(\lie{g}) \overline{\lie{n}}) 
\rightarrow U(\lie{m} \oplus \lie{a})
\overset{\mbox{\tiny rho shift}}{\longrightarrow} 
U(\lie{m} \oplus \lie{a}), 
\label{eq:Harish-Chandra map, n}\\
& 
\gamma_{\lie{u}}: 
U(\lie{m} \oplus \lie{a}) 
\simeq U(\lie{h}) 
\oplus (\lie{u} U(\lie{m} \oplus \lie{a}) 
+ U(\lie{m} \oplus \lie{a}) \overline{\lie{u}}) 
\rightarrow U(\lie{h})
\overset{\mbox{\tiny rho shift}}{\longrightarrow} 
U(\lie{h}). 
\end{align}
For notation, see \S\ref{section:introduction}. 
Then by \eqref{eq:main formula}, 
\begin{equation}\label{eq:image under gamma_n}
\gamma_{\lie{n}}(C_{n}(u)) 
= (u^{2}-H^{2}) C_{n-2}(u). 
\end{equation}
By the hypothesis of induction, $C_{n-2}(u)$ is a central element of
$U(\lie{so}_{n-2})$ and $\gamma_{\lie{u}}(C_{n-2}(u))$ is
$(u^{2}-T_{1}^{2}) \cdots 
(u^{2} - T_{\lfloor(n-2)/2\rfloor}^{2})$. 
Since the Harish-Chandra map $\gamma$ given in
\S\ref{section:introduction} is the composition 
$\gamma = \gamma_{\lie{u}} \circ \gamma_{\lie{n}}$, 
\eqref{eq:image of HC map} is shown. 

The image $\gamma(C_{n}(u))$ given in \eqref{eq:image of HC map} is %an
%element of $U(\lie{h})$ which is 
invariant under the action of the Weyl
group of $\lie{g}$. 
Therefore, there exists an element $z \in Z(\lie{g})$ such that
$\gamma(z) = \gamma(C_{n}(u))$. 
By \eqref{eq:image under gamma_n} and the hypothesis of induction, 
$\gamma_{\lie{n}}(C_{n}(u)-z)$ is an element of 
$Z(\lie{m} \oplus \lie{a})$. 
Since the restriction of $\gamma_{\lie{u}}$ to $Z(\lie{m} \oplus
\lie{a})$ is injective, $\gamma_{\lie{n}}(C_{n}(u) - z) = 0$. 

Consider the projection 
\[
p : 
U(\lie{g}) 
\simeq U(\lie{n}) \otimes U(\lie{a}) \otimes U(\lie{k}) 
\simeq (U(\lie{a}) \otimes U(\lie{k})) \oplus \lie{n} U(\lie{g}) 
\rightarrow U(\lie{a}) \otimes U(\lie{k}). 
\]
Recall the definition \eqref{eq:main formula} of $C_{n}(u)$. 
The right hand of \eqref{eq:image under gamma_n} is the same as 
$C_{n}(u) \mod \lie{n} \Ug$ composed by rho shift. 
It follows that 
$\gamma_{\lie{n}}(C_{n}(u)) = (\mbox{rho shift}) \circ p(C_{n}(u))$. 
By a result of Lepowsky's (\cite{L}), 
the restriction of
$p$ to the subalgebra 
$U(\lie{g})^{\lie{k}}$ of $\lie{k}$-invariants in $U(\lie{g})$ is
injective. 
Proposition~\ref{proposition:K-invariance of C_n} says that 
$C_{n}(u)$ is an element of $U(\lie{g})^{\lie{k}}$. 
It follows from $\gamma_{\lie{n}}(C_{n}(u) - z) = 0$ that 
$C_{n}(u)-z \in U(\lie{g})^{\lie{k}} \cap \Ker\, p 
= \{0\}$. 
This shows $C_{n}(u) = z \in \Zg$. 
\qed
%%%%%%%%%%%%%%

\section{A proof of Lemma~\ref{lemma:shift-central}}
\label{section:proof of lemma}

We shall prove Lemma~\ref{lemma:shift-central}, 
so we assume that $C_{n-2}(u) \in Z(\lie{so}_{n-2})$ and 
that it satisfies \eqref{eq:image of HC map} in this section. 
We first write $L(C_{n}(u))\phi_{\tau_{\lambda}}$ explicitly. 

For $Y_{1}, \dots, Y_{p} \in \lie{k}$, 
the action of $Y_{1} \cdots Y_{p} \in U(\lie{k})$ on 
$\phi_{\tau_{\lambda}}$ is given by 
\begin{align*}
& 
L(Y_{1} \cdots Y_{p}) \phi_{\tau_{\lambda}}(g) 
= 
(Y_{1} \cdots Y_{p})^{\mathrm{opp}}\, 
\phi_{\tau_{\lambda}}(g), 
&
&
(Y_{1} \cdots Y_{p})^{\mathrm{opp}} 
:= 
(-Y_{p}) \cdots (-Y_{1}). 
\end{align*}
Note that, as we have noticed, the symbol $\tau_{\lambda}$ is omitted. 
The map ``$\mathrm{opp}$'' satisfies 
\[
[A, B]^{\mathrm{opp}} 
= 
-[A^{\mathrm{opp}}, B^{\mathrm{opp}}]. 
\]
By the assumption, 
$C_{n-2}(u)$ and $C_{n-2}(u)^{\mathrm{opp}}$ are elements of
$Z(\lie{so}_{n-2})$. 
Since the shifted Harish-Chandra map $\gamma_{\lie{u}}$ does not
depend on the choice of $\lie{u}$, 
\begin{align*}
\gamma_{\lie{u}}(C_{n-2}(u)) 
&= 
\prod_{p=1}^{\lfloor(n-2)/2\rfloor}
(u^{2}-T_{p}^{2})
=\gamma_{\lie{u}}(C_{n-2}(u))^{\mathrm{opp}} 
\\
&= \gamma_{\overline{\lie{u}}}(C_{n-2}(u)^{\mathrm{opp}})
= \gamma_{\lie{u}}(C_{n-2}(u)^{\mathrm{opp}}).
\end{align*}
Therefore, 
$C_{n-2}(u)^{\mathrm{opp}} 
= 
C_{n-2}(u)$. 
Analogously, we have $\Omega_{n-2}^{\mathrm{opp}} = \Omega_{n-2}$. 
By the definition \eqref{eq:main formula} of $C_{n}(u)$, we have the
following lemma: 
\begin{lemma}\label{lemma:action of C_n}
Let 
$\phi_{\tau_{\lambda}} (g) 
= 
\sum_{Q \in GT(\lambda)} c(Q; g)\, Q$ be an element of  
$C_{\tau_{\lambda}}^{\infty}(K \backslash G)$. 
The action of $C_{n}(u)$ on it is given by 
\begin{align}
L(C_{n}(u))\, \phi_{\tau_{\lambda}}&(g)
\notag\\
=
\sum_{Q \in GT(\lambda)} 
\Bigg[ &
\left\{
-\left(L(H)-\frac{n-2}{2}\right)^{2} 
+ u^{2} 
- \sum_{i=1}^{n-2} L(X_{i})^{2} 
\right\} 
c(Q; g)\, C_{n-2}(u)\, Q 
\notag\\
&
+ \sum_{i=1}^{n-2} 
L(X_{i}) \left(L(H)-\frac{n-5}{2}\right) c(Q; g)\, 
[A_{n-1,i}, C_{n-2}(u)] \, Q 
\notag\\
& 
- 2 
\sum_{i=1}^{n-2} L(X_{i}) c(Q; g)\, A_{n-1,i} C_{n-2}(u)\, Q 
\label{eq:action of C_n(u)} 
\\
& + \frac{1}{2} 
\sum_{i=1}^{n-2} 
L(X_{i}) c(Q; g)\, 
[\Omega_{n-2}, [A_{n-1,i}, C_{n-2}(u)]]\, Q 
\notag
\\
& - \frac{1}{2} 
\sum_{i,j=1}^{n-2} 
L(X_{i}) L(X_{j}) c(Q; g)\, 
[A_{n-1,i}, [A_{n-1,j}, C_{n-2}(u)]]\, Q 
\Bigg].
\notag
\end{align}
\end{lemma}

In the proof of lemmas below, we treat the case when $\ell\not=0$. 
The difference between these cases 
and the case when $\ell=0$ is only one point: 
$a_{n-2,-\ell}(\sigma_{n-2,\ell} Q) 
= - a_{n-2,\ell}(Q)$ if $\ell\not=0$, 
but $a_{n-2,-0}(\sigma_{n-2,0}Q) = a_{n-2,0}(Q)$ 
if $n$ is even and $\ell=0$. 
If you modify this point, you get the proof of the latter case. 

Next, we see the relationship between 
$\varpi_{-\ell} \varpi_{\ell}$ and $C(u_{\ell})$.

\begin{lemma}\label{lemma:pi pi = C_{n-2}}
There exists a non-zero constant $d_{\lambda,\ell}$ which does not
depend on the 
$\vect{q}_{1}, \dots, \vect{q}_{n-3}$ parts of 
$Q = (\vect{q}_{1}, \dots, \vect{q}_{n-3}, \vect{q}_{n-2}) 
\in GT(\lambda)$ such that 
\begin{equation}\label{eq:pi pi = C_{n-2}}
\varpi_{-\ell} \varpi_{\ell} Q 
= -d_{\lambda,\ell}\, C_{n-2}(u_{\ell})\, Q. 
\end{equation}
\end{lemma}
\begin{proof}
By the definition of the Gelfand-Tsetlin basis and $\varpi_{\ell}$, 
the action of $\varpi_{-\ell} \varpi_{\ell}$ on $Q \in GT(\lambda)$ is
given by 
\begin{align}
\varpi_{-\ell} \varpi_{\ell} Q 
&= 
a_{n-2,-\ell}(\sigma_{n-2,\ell}Q) 
\sigma_{n-2,-\ell} \, 
a_{n-2,\ell}(Q)\, 
\sigma_{n-2,\ell} Q 
= 
-a_{n-2,\ell}(Q)^{2}\, Q 
\notag\\
&= 
-d_{\lambda,\ell} 
\prod_{1 \leq |i| \leq \lfloor (n-2)/2 \rfloor} 
(l_{n-2,\ell} + l_{n-3,i}) \, Q,
\label{eq:action of pi pi}
\end{align}
where $d_{\lambda,\ell}$ is a constant which does not depend on the 
$\vect{q}_{1}, \dots, \vect{q}_{n-3}$ parts of 
$Q = (\vect{q}_{1}, \dots, \vect{q}_{n-3}, \vect{q}_{n-2}) 
\in GT(\lambda)$. 

On the other hand, 
$C_{n-2}(u)$ acts on $Q$ by a scalar, 
since 
$Q = (\vect{q}_{1}, \dots, \vect{q}_{n-3}, \vect{q}_{n-2})$ is
contained in the irreducible representation of $SO(n-2)$ with 
highest weight $\vect{q}_{n-3}$ 
(cf. Remark~\ref{remark:GT restriction}) 
and 
$C_{n-2}(u)$ is an element of $Z(\lie{so}_{n-2})$.

Let us calculate this scalar. 
By the assumption, 
\[
\gamma_{\lie{u}}(C_{n-2}(u))
=
\prod_{i=1}^{\lfloor (n-2)/2 \rfloor} 
(u^{2} - T_{i}^{2}). 
\]
The ``rho'' of $\lie{so}_{n-2}$ is 
$\rho_{\lie{so}_{n-2}}
:= \frac{1}{2} 
\sum_{i=1}^{\lfloor (n-2)/2 \rfloor} 
(n-2-2i) e_{i}$. 
It follows that 
\begin{equation}
C_{n-2}(u)\, Q 
= 
\prod_{i=1}^{\lfloor (n-2)/2 \rfloor} 
\left\{u^{2} - \left(q_{n-3,i}+\frac{n-2-2i}{2}\right)^{2}\right\}
\, Q. 
\label{eq:action of C_{n-2}(u)-1}
\end{equation}

When $n=2m+1$ is odd and $i>0$, 
then $q_{n-3,i}+(n-2-2i)/2 = l_{n-3,i}-1/2$ 
and $l_{n-3,-i} = 1 - l_{n-3,i}$. 
Therefore 
\begin{align}
\prod_{1 \leq |i| \leq \lfloor (n-2)/2 \rfloor} 
(l_{n-2,\ell} + l_{n-3,i}) 
&=
\prod_{i=1}^{m-1} 
(l_{n-2,\ell} + l_{n-3,i}) (l_{n-2,\ell} - l_{n-3,i} + 1) 
\notag\\
&=
\prod_{i=1}^{m-1} 
\left\{\left(l_{n-2,\ell} + \frac{1}{2}\right)^{2} 
- \left(q_{n-3,i}+\frac{n-2-2i}{2}\right)^{2}\right\}. 
\label{eq:action of C_{n-2}(u)-2}
\end{align}

When $n=2m$ is even and $i>0$, 
then $q_{n-3,i}+(n-2-2i)/2 = l_{n-3,i}$ 
and $l_{n-3,-i} = - l_{n-3,i}$. 
Therefore 
\begin{align}
\prod_{1 \leq |i| \leq \lfloor (n-2)/2 \rfloor} 
(l_{n-2,\ell} + l_{n-3,i}) 
&=
\prod_{i=1}^{m-1}
(l_{n-2,\ell} + l_{n-3,i}) (l_{n-2,\ell} - l_{n-3,i}) 
\notag\\
&=
\prod_{i=1}^{m-1}
\left\{(l_{n-2,\ell})^{2} - \left(q_{n-3,i}+\frac{n-2-2i}{2}\right)^{2}
\right\}. 
\label{eq:action of C_{n-2}(u)-3}
\end{align}
Then \eqref{eq:pi pi = C_{n-2}} follows from 
\eqref{eq:action of pi pi}, 
\eqref{eq:action of C_{n-2}(u)-1}, \eqref{eq:action of C_{n-2}(u)-2}, 
\eqref{eq:action of C_{n-2}(u)-3} and 
Definition~\ref{definition:u_l}. 
If $d_{\lambda,\ell}$ is not zero, then the lemma is proved. 

Consider the case when $d_{\lambda,\ell}$ 
in \eqref{eq:action of pi pi} is zero. 
By the definition of $a_{n-2,\ell}(Q)$ and \eqref{eq:action of pi pi}, 
$d_{\lambda,\ell}$ is zero if and only if one of the following
conditions is satisfied: 
\begin{enumerate}
\item
$1 < \ell \leq \lfloor n/2 \rfloor$ 
and $\lambda_{\ell} = \lambda_{\ell-1}$. 
\item
$-\lfloor n/2 \rfloor + 1 \leq \ell \leq -1$ 
and $\lambda_{|\ell|} = \lambda_{|\ell|+1}$. 
\item
$n=2m+1$ is odd, $\lambda_{m-1}=-\lambda_{m}$ and 
$\ell=-m+1, -m$. 
\end{enumerate}

In the case (1), $q_{n-3,\ell-1}=\lambda_{\ell}$ 
since $\lambda_{\ell-1} \geq q_{n-3,\ell-1} \geq \lambda_{\ell}$. 
By Definition~\ref{definition:u_l}, 
\[
q_{n-3,\ell-1}+\{n-2-2(\ell-1)\}/2
=
\lambda_{\ell}+n/2-\ell 
= u_{\ell}.
\]
In the case (2), 
$q_{n-3,|\ell|}=\lambda_{|\ell|}$ 
since $\lambda_{|\ell|} \geq q_{n-3, |\ell|} \geq \lambda_{|\ell|+1}$. 
By Definition~\ref{definition:u_l}, 
\[
q_{n-3,|\ell|}+(n-2-2 |\ell|)/2
=
(\lambda_{|\ell|}+n/2-|\ell|)-1 
= -u_{\ell}
\]
In the case (3), 
$\lambda_{m-1} = q_{2m-2,m-1} = -\lambda_{m}$ 
since 
$\lambda_{m-1} \geq q_{2m-2,m-1} \geq -\lambda_{m}$. 
The numbers $u_{-m+1}$, $u_{-m}$ and $q_{n-3,m-1}+\{n-2-2(m-1)\}/2$ are 
\begin{align*}
& 
u_{-m+1} = 1-\{\lambda_{m-1}+(2m+1)/2-m+1\} 
= -(\lambda_{m-1}+1/2), 
\\
&
u_{-m} = 1-\{\lambda_{m}+(2m+1)/2-m\} 
=-\lambda_{m}+1/2 \quad \mbox{and}
\\
& q_{n-3,m-1}+\frac{2m+1-2-2(m-1)}{2}
=q_{2m-2,m-1}+1/2
=-u_{-m+1} = u_{-m}. 
\end{align*}
In every case, we get $C_{n-2}(u_{\ell}) Q = 0$ 
by \eqref{eq:action of C_{n-2}(u)-1}. 

On the other hand, if $d_{\lambda,\ell}$ in \eqref{eq:action of pi pi}
is zero, then $\varpi_{-\ell} \varpi_{\ell} Q = 0$. 
Therefore, if we relpace $d_{\lambda,\ell}$ by a non-zero constant, 
then \eqref{eq:pi pi = C_{n-2}} holds. 
\end{proof}

In the reminder of this section, we show that \eqref{eq:P_{-k} P_k} and 
$d_{\lambda,\ell} \times$\eqref{eq:action of C_n(u)} are identical
when $u = u_{\ell}$. 
We first show that the terms which do not contain $L(X_{i})$ in
these are identical. 
\begin{lemma}\label{lemma:no X}
For $Q \in GT(\lambda)$, 
\begin{align}
& 
\left(L(H) - l_{n-2,\ell}-\lfloor \frac{n-1}{2} \rfloor\right)
\left(L(H) + l_{n-2,\ell}-\lfloor \frac{n-2}{2} \rfloor\right)\, 
\varpi_{-\ell} \varpi_{\ell}\, Q 
\label{eq:not contain L(X)}
\\
&\qquad 
=
- d_{\lambda,\ell} 
\left\{\left(L(H)-\frac{n-2}{2}\right)^{2} 
- u_{\ell}^{2}
\right\} 
C_{n-2}(u_{\ell})\, Q.
\notag
\end{align}
\end{lemma}
\begin{proof}
By Definition~\ref{definition:u_l} and 
\begin{align*}
& 
\lfloor \frac{n-1}{2} \rfloor
=
\begin{cases}
\frac{n-1}{2} \quad \mbox{if $n$ is odd} 
\\
\frac{n-2}{2} \quad \mbox{if $n$ is even}, 
\end{cases}
& 
\lfloor \frac{n-2}{2} \rfloor
=
\begin{cases}
\frac{n-3}{2} \quad \mbox{if $n$ is odd} 
\\
\frac{n-2}{2} \quad \mbox{if $n$ is even}, 
\end{cases}
\end{align*}
we have 
\begin{align}
&
- l_{n-2,\ell}-\lfloor \frac{n-1}{2} \rfloor
= -\frac{n-2}{2} - u_{\ell}, 
&
&
l_{n-2,\ell}-\lfloor \frac{n-2}{2} \rfloor
= -\frac{n-2}{2} + u_{\ell}. 
\label{eq:floor of (n-2)/2}
\end{align}
Therefore, this lemma follows from \eqref{eq:pi pi = C_{n-2}}. 
\end{proof}

Next, we check the terms containing $L(X_{i})\, L(X_{j})$, 
$1 \leq i \leq j \leq n-2$. 
We know $X_{i}$ and $X_{j}$ commute. 
Moreover, 
\begin{align*}
[A_{n-1,i},\,[A_{n-1,j},\,C_{n-2}(u)]]
&= 
[A_{j,i}, \, C_{n-2}(u)]
+
[A_{n-1,j},\,[A_{n-1,i},\,C_{n-2}(u)]]
\\
&= 
[A_{n-1,j},\,[A_{n-1,i},\,C_{n-2}(u)]],
\end{align*}
since $A_{j,i} \in \lie{so}_{n-2}$ and 
$C_{n-2}(u) \in Z(\lie{so}_{n-2})$ by the assumption. 
Therefore, what we should show is the following lemma: 

\begin{lemma}\label{lemma:X^2}
For $Q \in GT(\lambda)$, 
\begin{align}
[\varpi_{-\ell}, &A_{n-1,i}] [\varpi_{\ell}, A_{n-1,j}] Q
+
[\varpi_{-\ell}, A_{n-1,j}] [\varpi_{\ell}, A_{n-1,i}] Q
\notag\\
&= 
- d_{\lambda,\ell} 
\{
2 \delta_{i,j}\, 
C_{n-2}(u)\, 
+ 
[A_{n-1,i}, [A_{n-1,j}, C_{n-2}(u)]]
\} Q.
\label{eq:contains L(X)^2} 
\end{align}
\end{lemma}
\begin{proof}
By \eqref{eq:pi pi = C_{n-2}}, we have 
\begin{align*}
&[A_{n-1,i},\, [A_{n-1,j},\, -d_{\lambda,\ell}\, C_{n-2}(u_{\ell})]] Q
\notag\\
&=
[A_{n-1,i},\, [A_{n-1,j},\, \varpi_{-\ell} \varpi_{\ell} ]] Q
\\
&=
[A_{n-1,i},\, [A_{n-1,j},\, \varpi_{-\ell}]] \varpi_{\ell} Q 
+ 
[A_{n-1,j},\, \varpi_{-\ell}] [A_{n-1,i},\, \varpi_{\ell}] Q 
\\
& \qquad 
+ 
[A_{n-1,i},\, \varpi_{-\ell}] [A_{n-1,j},\, \varpi_{\ell}] Q 
+ 
\varpi_{-\ell} [A_{n-1,i},\, [A_{n-1,j},\, \varpi_{\ell}]] Q.
\end{align*}
As we remarked in Remark~\ref{remark:varpi}, $\varpi_{\ell}Q$ is
identified with the 
$V_{\lambda+e_{\ell}}$ component of 
$\tau_{\widetilde{\lambda}}(A_{n,n-1}) Q$. 
Therefore, 
$[A_{n-1,i},\, [A_{n-1,j},\, \varpi_{\pm\ell}]]Q$ is identified with the
$V_{\lambda+e_{\pm\ell}}$ component of 
\begin{align*}
&[A_{n-1,i},\, [A_{n-1,j},\, A_{n,n-1}]]Q
= -[A_{n-1,i}, A_{n,j}]Q
= - \delta_{i,j}\, A_{n,n-1}Q 
\\
&= -\delta_{i,j} 
\sum_{k} a_{n-2,k}(Q) \sigma_{n-2,k}Q 
= 
-\delta_{i,j} 
\sum_{k} \varpi_{k} Q, 
\end{align*}
namely, identified with $- \delta_{i,j}\, \varpi_{\pm \ell}$. 
Then we get 
\begin{align*}
&[A_{n-1,i},\, [A_{n-1,j},\, \varpi_{-\ell}]] 
\varpi_{\ell} Q 
=
\varpi_{-\ell} 
[A_{n-1,i},\, [A_{n-1,j},\, \varpi_{\ell}]] Q 
\\
&= 
-\delta_{i,j}\, \varpi_{-\ell} \varpi_{\ell} Q 
= 
\delta_{i,j}\, d_{\lambda,\ell} C_{n-2}(u_{\ell}) Q. 
\end{align*}
It follows that 
\begin{align*}
&[A_{n-1,i},\, \varpi_{-\ell}] [A_{n-1,j},\, \varpi_{\ell}] Q 
+ [A_{n-1,j},\, \varpi_{-\ell}] [A_{n-1,i},\, \varpi_{\ell}] Q 
\\
&=
-2 \delta_{i,j}\, 
d_{\lambda,\ell} C_{n-2}(u_{\ell}) Q 
-d_{\lambda,\ell}\, 
[A_{n-1,n-2}, [A_{n-1,n-2}, C_{n-2}(u_{\ell})]] Q,
\end{align*}
so \eqref{eq:contains L(X)^2} holds. 
\end{proof}

Finally, we check the terms containing $L(X_{i})$, $i=1,\dots,n-2$. 

Consider the actions of $M\simeq SO(n-2)$ on 
$\lier{n} = \R\mbox{-span}\{X_{i} \,|\, i=1,2,\dots,n-2\}$ and 
on $\R\mbox{-span}\{A_{n-1,i}\,|\, i=1,2,\dots,n-2\}$. 
We can find elements $m_{i} \in M$ ($1 \leq i \leq n-2$) 
such that 
\begin{align*}
& \Ad(m_{i}) X_{n-2} = X_{i} 
& & \mbox{and} 
&
&
\Ad(m_{i}) A_{n-1,n-2} = A_{n-1,i}. 
\end{align*}
If $Y_{1}, Y_{2}$ are $M$-invariant elements in $\Ug$, 
then 
\begin{align*}
& 
\Ad(m_{i})
(X_{n-2} Y_{1} A_{n-1,n-2} Y_{2}) 
= 
X_{i} Y_{1} A_{n-1,i} Y_{2}. 
\end{align*}

By the definition \eqref{eq:varpi} of $\varpi_{\ell}$, 
its action commutes with $m \in M$. 
Morover, $\Omega_{n-2}$ is $M$-invariant; 
so is $C_{n-2}(u)$ by the assumption. 
It follows that, if we can show the terms containing $L(X_{n-2})$
in \eqref{eq:P_{-k} P_k} and 
$d_{\lambda,\ell} \times$\eqref{eq:action of C_n(u)} are identical for
$u = u_{\ell}$, 
then the terms containing $L(X_{i})$, $i=1,\dots, n-3$, 
are also identical. 
Therefore, the next lemma will complete the proof of
Lemma~\ref{lemma:shift-central}. 
\begin{lemma}\label{lemma:X^1}
For $Q \in GT(\lambda)$, 
\begin{align}
&
\left(
L(H) - l_{n-2,\ell}-\lfloor\frac{n-3}{2}\rfloor
\right) %c(Q; g) 
\varpi_{-\ell} [\varpi_{\ell}, A_{n-1,n-2}] Q
\notag\\
& \quad 
+
\left(
L(H)+l_{n-2,\ell}-\lfloor \frac{n-2}{2}\rfloor
\right) 
[\varpi_{-\ell}, A_{n-1,n-2}] \varpi_{\ell} Q
\notag\\
&= 
d_{\lambda,\ell} 
\Bigg\{
\left(L(H)-\frac{n-5}{2}\right) 
[A_{n-1,n-2}, C_{n-2}(u_{\ell})] 
\label{eq:X^1}\\
& \qquad \qquad 
- 2  A_{n-1,n-2} C_{n-2}(u_{\ell}) 
+ \frac{1}{2} 
[\Omega_{n-2}, [A_{n-1,n-2}, C_{n-2}(u_{\ell})]]
\Bigg\} Q. 
\notag
\end{align}
\end{lemma}
\begin{proof}
By \eqref{eq:pi pi = C_{n-2}} and \eqref{eq:floor of (n-2)/2}, 
the difference of both sides of \eqref{eq:X^1} is 
%what we have to show is that 
\begin{align}
& \left(L(H) - u_{\ell} - \frac{n-4}{2}\right) 
\varpi_{-\ell} [\varpi_{\ell}, A_{n-1,n-2}] Q
\notag\\
& \quad 
+
\left(L(H) + u_{\ell} - \frac{n-2}{2}\right) 
[\varpi_{-\ell}, A_{n-1,n-2}] \varpi_{\ell} Q
\notag\\
& \quad 
+
\left(L(H)-\frac{n-5}{2}\right) 
[A_{n-1,n-2}, \varpi_{-\ell} \varpi_{\ell}] \, Q 
\notag\\
& \quad 
- 2 A_{n-1,n-2} \varpi_{-\ell} \varpi_{\ell}\, Q 
+ \frac{1}{2} 
[\Omega_{n-2}, [A_{n-1,n-2}, \varpi_{-\ell} \varpi_{\ell}]]\, Q
\notag\\
&=
\left(u_{\ell} + \frac{1}{2}\right) 
\varpi_{-\ell} [A_{n-1,n-2}, \varpi_{\ell}] Q
+
\left(-u_{\ell} + \frac{3}{2}\right) 
[A_{n-1,n-2}, \varpi_{-\ell}] \varpi_{\ell} Q
\label{eq:last eq of X^{1}}\\
& \quad 
- 2 A_{n-1,n-2} \varpi_{-\ell} \varpi_{\ell}\, Q 
+ \frac{1}{2} 
[\Omega_{n-2}, [A_{n-1,n-2}, \varpi_{-\ell} \varpi_{\ell}]] Q.
\notag
\end{align}
We shall show that this is zero. 
For simplicity, we denote 
\begin{align*}
& A:=A_{n-1,n-2}, 
&
& \Omega_{n-2} := \Omega, 
&
& a_{\ell}(Q) := a_{n-2,\ell}(Q), 
\\
& a_{j}(Q) := a_{n-3,j}(Q), 
&
& \sigma_{\ell} := \sigma_{n-2,\ell} 
&
& \sigma_{j} := \sigma_{n-3,j}, 
\\
& l_{\ell} := l_{n-2,\ell} 
\quad \mbox{and} \quad 
&
& l_{j} := l_{n-3,j}. 
\end{align*}
By the definitions of $\varpi_{\ell}$ and the Gelfand-Tsetlin basis, 
\begin{align*}
&
\varpi_{\ell} Q = a_{\ell}(Q) \sigma_{\ell} Q, 
&
&
A Q = \sum_{j} a_{j}(Q) \sigma_{j}Q 
\end{align*}
for $Q \in GT(\lambda)$. 
By the definition of $a_{\ell}(Q)$ and $a_{j}(Q)$, 
we have 
\[
\frac{a_{\ell}(Q) a_{j}(\sigma_{\ell}Q)}
{a_{j}(Q) a_{\ell}(\sigma_{j}Q)} 
= 
\frac{l_{\ell}+l_{-j}}{l_{\ell}+l_{-j}-1}
\quad  \mbox{and} \quad  
a_{-\ell}(\sigma_{\ell}Q) = - a_{\ell}(Q). 
\]
It follows that 
\begin{align}
\varpi_{-\ell} [A, \varpi_{\ell}] Q 
&= 
\sum_{j} a_{-\ell}(\sigma_{j} \sigma_{\ell}Q) 
\{a_{\ell}(Q) a_{j}(\sigma_{\ell}Q) - a_{j}(Q) a_{\ell}(\sigma_{j}Q)\} 
\sigma_{j} Q 
\notag\\
&= 
-\sum_{j} 
\frac{a_{\ell}(\sigma_{j}Q)^{2} a_{j}(Q)}{l_{\ell} + l_{-j} - 1} 
\sigma_{j} Q. 
\label{eq:pi-A-pi-1}
\end{align}
By analogous calculations, we obtain 
\begin{align}
& 
[A, \varpi_{-\ell}] \varpi_{\ell} Q 
= 
\sum_{j} 
\frac{a_{\ell}(Q)^{2} a_{j}(Q)}{l_{\ell} + l_{j}} 
\sigma_{j} Q, 
\label{eq:pi-A-pi-2}
\\
& 
A \varpi_{-\ell} \varpi_{\ell} Q 
= 
-\sum_{j} a_{\ell}(Q)^{2} a_{j}(Q) \sigma_{j} Q, 
\label{eq:pi-A-pi-3}
\\
& 
[A, \varpi_{-\ell} \varpi_{\ell}] Q 
= 
\sum_{j} \{a_{\ell}(\sigma_{j} Q)^{2} - a_{\ell}(Q)^{2}\} 
a_{j}(Q) \sigma_{j} Q. 
\label{eq:pi-A-pi-4}
\end{align}
Since the ``rho'' of $\lie{so}_{n-2}$ is 
$\rho_{\lie{so}_{n-2}} = \frac{1}{2} 
\sum_{i=1}^{\lfloor (n-2)/2 \rfloor} 
(n-2-2i) e_{i}$, 
$\Omega$ acts on $Q$ by the scalar 
\[
-|
\vect{q}_{n-3} + \rho_{\lie{so}_{n-2}}|^{2} 
+ 
|\rho_{\lie{so}_{n-2}}|^{2} 
= 
-\sum_{i=1}^{\lfloor \frac{n-2}{2} \rfloor}
\left\{
\left(\frac{l_{i} - l_{-i}}{2}\right)^{2} 
- 
|\rho_{\lie{so}_{n-2}}|^{2} 
\right\}.
\]
It follows that 
\begin{align*}
[\Omega,\, &[A,\, \varpi_{-\ell} \varpi_{\ell}]] Q 
\\
&= \sum_{j} \{a_{\ell}(\sigma_{j} Q)^{2} - a_{\ell}(Q)^{2}\} 
a_{j}(Q) 
\left\{
\left(\frac{l_{j} - l_{-j}}{2}\right)^{2} 
- 
\left(\frac{l_{j} - l_{-j}}{2} + 1\right)^{2} 
\right\} \sigma_{j} Q 
\\
&= 
\sum_{j} \{a_{\ell}(Q)^{2} - a_{\ell}(\sigma_{j} Q)^{2}\} 
a_{j}(Q) (l_{j} - l_{-j} + 1) \sigma_{j} Q. 
\end{align*}
By \eqref{eq:pi-A-pi-1}, \eqref{eq:pi-A-pi-2}, 
\eqref{eq:pi-A-pi-3}, \eqref{eq:pi-A-pi-4} 
and 
\[
\frac{a_{\ell}(\sigma_{j}Q)^{2}}{a_{\ell}(Q)^{2}} 
= 
\frac{(l_{\ell}+l_{j}+1)(l_{\ell}+l_{-j}-1)}
{(l_{\ell}+l_{j})(l_{\ell}+l_{-j})},
\]
the coefficient of $\sigma_{j} Q$ in \eqref{eq:last eq of X^{1}} is 
\begin{align*}
& 
-\left(u_{\ell} + \frac{1}{2}\right) 
\frac{a_{\ell}(\sigma_{j}Q)^{2} a_{j}(Q)}{l_{\ell} + l_{-j} - 1} 
+
\left(-u_{\ell} + \frac{3}{2}\right) 
\frac{a_{\ell}(Q)^{2} a_{j}(Q)}{l_{\ell} + l_{j}} 
\\
& \quad + 2 a_{\ell}(Q)^{2} a_{j}(Q) 
+ \frac{1}{2} 
\{a_{\ell}(Q)^{2} - a_{\ell}(\sigma_{j} Q)^{2}\} 
a_{j}(Q) (l_{j} - l_{-j} + 1) 
\\
&= 
\frac{a_{\ell}(Q)^{2} a_{j}(Q)} 
{(l_{\ell}+l_{j})(l_{\ell}+l_{-j})}
\\
& 
\quad \times 
\Bigl\{
-\left(u_{\ell} + \frac{1}{2}\right) (l_{\ell}+l_{j}+1) 
+
\left(-u_{\ell} + \frac{3}{2}\right) (l_{\ell}+l_{-j}) 
+ 2 (l_{\ell}+l_{j})(l_{\ell}+l_{-j}) 
\\
& \qquad \quad 
+ \frac{1}{2}
\Bigl((l_{\ell}+l_{j})(l_{\ell}+l_{-j}) 
-
(l_{\ell}+l_{j}+1)(l_{\ell}+l_{-j}-1)\Bigr)
(l_{j} - l_{-j} + 1)\Bigr\}
\end{align*}
If $n$ is odd, then $u_{\ell} = l_{\ell}  +1/2$ and 
$l_{-j} = 1 - l_{j}$. 
In this case, the term in the braces is 
\[
-(l_{\ell}+1)(l_{\ell}+l_{j}+1) + (-l_{\ell}+1)(l_{\ell}-l_{j}+1) 
+ 2(l_{\ell}+l_{j})(l_{\ell}-l_{j}+1) + 2(l_{j})^{2} 
= 0. 
\]
If $n$ is even, then $u_{\ell} = l_{\ell}$ and $l_{-j} = - l_{j}$. 
In this case, the term in the braces is 
\[
-(l_{\ell}+\frac{1}{2})(l_{\ell}+l_{j}+1) 
+ (-l_{\ell} + \frac{3}{2})(l_{\ell} - l_{j}) 
+ 2(l_{\ell}+l_{j})(l_{\ell}-l_{j}) + \frac{1}{2} (2 l_{j}+1)^{2} 
= 0.
\]
Therefore, \eqref{eq:last eq of X^{1}} is zero. 
\end{proof}

\section{Pfaffian}
\label{section:Pfaffian}

When $n=2m$ is even, there is another generator of $Z(\lie{so}_{2m})$,
which is called Pfaffian. 
In this section, we obtain the Iwasawa decomposition of this
element and relate it to the $K$-type shift operator $P_{0}$. 

For a set $\{i_{1}, i_{2}, \dots, i_{2k}\}$ of $2k$ different positive
integers, define the Pfaffian 
$\Pf_{2k}(i_{2k},i_{2k-1},\dots,i_{1})$ inductively by 
\begin{align*}
&
\Pf_{2}(i_{2}, i_{1}) = A_{i_{2},i_{1}}, 
\\
&
\Pf_{2k}(i_{2k},i_{2k-1},\dots,i_{1}) 
= 
\sum_{j=1}^{2k-1} (-)^{j+1} A_{i_{2k},i_{j}} 
\Pf_{2k-2}(i_{2k-1},\dots,\widehat{i_{j}},\dots,i_{1}),
\intertext{and define}
&
\PF_{2k} = \Pf_{2k}(2k,2k-1,\dots,1). 
\end{align*}

\begin{lemma}
\begin{enumerate}
\item
For every permutation $\sigma \in S_{2k}$, 
\begin{equation}\label{eq:Pfaffian, symmetry}
\Pf_{2k}(i_{\sigma(2k)}, i_{\sigma(2k-1)}, \dots, i_{\sigma(1)}) 
= 
\sgn \sigma \, \Pf_{2k}(i_{2k}, i_{2k-1}, \dots, i_{1}). 
\end{equation}
\item
$\PF_{2m}$ is an element of $Z(\lie{so}_{2m})$. 
\end{enumerate}
\end{lemma}
\begin{proof}
(1) It is enough to show that $\PF_{2k}$ satisfies 
\eqref{eq:Pfaffian, symmetry}. We will show it by induction on $k$. 

If $k=1$, then $\PF_{2} = A_{2,1}$ is alternative under the action of
$S_{2}$. 
Assume that \eqref{eq:Pfaffian, symmetry} holds for $k - 1$. 
Then by the definition 
\begin{align*}
\PF_{2k} 
&= 
\sum_{i=1}^{2k-1} 
(-)^{i+1} A_{2k,i} \Pf_{2k-2}(2k-1,\dots,\hat{i},\dots,1)
%\label{eq:Pfaffian, expansion}
\end{align*}
and the hypothesis of induction, 
we can easily show that \eqref{eq:Pfaffian, symmetry} holds for 
adjacent transpositions $\sigma = (j, j+1)$, $j=1,\dots,2k-2$. 

Next, consider the case when $\sigma$ is the transposition $(2k-1,2k)$. 
By definition, 
\begin{align}
\PF_{2m} 
&= 
A_{2m,2m-1} \Pf_{2m-2}(2m-2,\dots,1) 
\notag\\
& \quad +
\sum_{i=1}^{2m-1} 
\sum_{j=i+1}^{2m-2} 
(-)^{i+1} (-)^{j} A_{2m,i} A_{2m-1,j} 
\Pf_{2m-4}(2m-2,\dots,\hat{j},\dots,\hat{i},\dots,1) 
\notag\\
&\quad + 
\sum_{i=1}^{2m-1} 
\sum_{j=1}^{i-1} 
(-)^{i+1} (-)^{j-1} A_{2m,i} A_{2m-1,j} 
\Pf_{2m-4}(2m-2,\dots,\hat{i},\dots,\hat{j},\dots,1) 
\notag\\
&= 
A_{2m,2m-1} \Pf_{2m-2}(2m-2,\dots,1) 
\notag\\
&\quad + 
\sum_{1 \leq i < j \leq 2m-1}
(-)^{i+j+1} (A_{2m,i} A_{2m-1,j} - A_{2m,j} A_{2m-1,i}) 
\label{eq:Pfaffian, expansion twice}\\
& \qquad \qquad \qquad \qquad \times 
\Pf_{2m-4}(2m-2,\dots,\hat{j},\dots,\hat{i},\dots,1). 
\notag
\end{align}
By replacing $m$ with $k$, 
we get $\Pf_{2k}(2k-1,2k,2k-2,\dots,1) = - \PF_{2k}$.  
Since $\PF_{2k}$ is alternative under the action of all adjacent
transpositions in $S_{2k}$, 
it satisfies \eqref{eq:Pfaffian, symmetry}. 

(2) By (1), 
it is enough to show that $\PF_{2m}$ and $A_{2m,2m-1}$ commute, 
but this is clear from \eqref{eq:Pfaffian, expansion twice}. 
\end{proof}

Let us consider the Iwasawa decomposition of $\PF_{2m}$. 
Substitute 
\begin{align*}
& 
H = \I A_{2m,2m-1}, 
&
&
X_{i} = A_{2m-1,i} + \I A_{2m,i} \quad (i=1,2,\dots,2m-2)
\end{align*}
into \eqref{eq:Pfaffian, expansion twice}. 
Then we get 
\begin{align*}
\I \PF_{2m} 
&= 
H\, \PF_{2m-2} 
+ \sum_{i=1}^{2m-2} 
(-)^{i+1} X_{i}\, \Pf_{2m-2}(2m-1,\dots,\hat{i},\dots,1) 
\\
&\quad + 
\sum_{1 \leq i < j \leq 2m-1}
(-)^{i+j} [A_{2m-1,i}, A_{2m-1,j}] 
\\
& \qquad \qquad \qquad \qquad \times 
\Pf_{2m-4}(2m-2,\dots,\hat{j},\dots,\hat{i},\dots,1) 
\\
&= 
H\, \PF_{2m-2} 
+ \sum_{i=1}^{2m-2} 
(-)^{i+1} X_{i}\, \Pf_{2m-2}(2m-1,\dots,\hat{i},\dots,1) 
\\
&\quad + 
\sum_{1 \leq i < j \leq 2m-1}
(-)^{i+j} A_{j,i}\, 
\Pf_{2m-4}(2m-2,\dots,\hat{j},\dots,\hat{i},\dots,1). 
\end{align*}
On the other hand,  
\begin{align*}
&(2m-2) \PF_{2m-2} 
\\
&= 
\sum_{i=1}^{2m-2} (-)^{i}\, \Pf_{2m-2}(i,2m-2,\dots,\hat{i},\dots,1) 
\\
&= \sum_{i=1}^{2m-2} (-)^{i} 
\sum_{j=i+1}^{2m-2} (-)^{j} A_{i,j}\, 
\Pf_{2m-4}(2m-2,\dots,\hat{j},\dots,\hat{i},\dots,1)
\\
& \qquad 
+ 
\sum_{i=1}^{2m-2} (-)^{i} 
\sum_{j=1}^{i-1} (-)^{j-1} A_{i,j}\, 
\Pf_{2m-4}(2m-2,\dots,\hat{i},\dots,\hat{j},\dots,1)
\\
&=
-2 \sum_{1 \leq i < j \leq 2m-2} 
(-)^{i+j} 
A_{j,i}\, \Pf_{2m-2}(2m-2,\dots,\hat{j},\dots,\hat{i},\dots,1).
\end{align*}
Moreover, 
\begin{align*}
& 
[A_{2m-1,i},\, \PF_{2m-2}]
\\
&=
[A_{2m-1,i},\, (-)^{i}\ \Pf_{2m-2}(i,2m-2,\dots,\hat{i},\dots,1)]
\\
&= 
[A_{2m-1,i},\, (-)^{i}\ \sum_{j=i+1}^{2m-2} 
(-)^{j} A_{i,j}\, \Pf_{2m-4}(2m-2,\dots,\hat{j},\dots,\hat{i},\dots,1)]
\\
& \quad 
+ [A_{2m-1,i},\, 
(-)^{i}\ \sum_{j=1}^{i-1} 
(-)^{j+1} A_{i,j}\, \Pf_{2m-4}(2m-2,\dots,\hat{i},\dots,\hat{j},\dots,1)
\\
&= 
(-)^{i}\ \sum_{j=i+1}^{2m-2} 
(-)^{j} A_{2m-1,j}\, \Pf_{2m-4}(2m-2,\dots,\hat{j},\dots,\hat{i},\dots,1)]
\\
& \quad 
+ 
(-)^{i}\ \sum_{j=1}^{i-1} 
(-)^{j+1} A_{2m-1,j}\, \Pf_{2m-4}(2m-2,\dots,\hat{i},\dots,\hat{j},\dots,1)
\\
&=
(-)^{i} 
\Pf_{2m-2}(2m-1,\dots,\hat{i},\dots,1). 
\end{align*}
Therefore, we get the following Iwasawa decomposition of $\PF_{2m}$: 
\begin{proposition}
\begin{equation}\label{eq:Iwaswa decomposition of Pfaffian}
\I \PF_{2m}
= 
(H - m + 1)\, \PF_{2m-2} 
- \sum_{i=1}^{2m-2} 
X_{i}\, [A_{2m-1,i},\PF_{2m-2}].
\end{equation}
\end{proposition}

Let us calculate the action of $\PF_{2m}$ on 
$C_{\tau_{\lambda}}^{\infty}(K \backslash G)$. 
We use the Cartan subalgebra 
$\lie{h} = \lie{t}_{\lie{m}} \oplus \lie{a}$, 
its basis $H, T_{1}, \dots, T_{m-1}$ and the dual basis 
$\{\alpha, e_{1}, \dots, e_{m-1}\}$ 
defined in \S\ref{section:introduction}. 

By \eqref{eq:Iwaswa decomposition of Pfaffian} and 
\eqref{eq:Harish-Chandra map, n}, 
\[
\I \gamma_{\lie{n}}(\PF_{2m}) 
= 
H \PF_{2m-2}. 
\]
By induction on $m$, the image of the (shifted) Harish-Chandra map is
\begin{equation}\label{eq:HC map image of Pfaffian}
(\I)^{m} \gamma(\PF_{2m}) 
= 
H T_{1} \cdots T_{m-2} T_{m-1}. 
\end{equation}

Suppose that $Q = (\vect{q}_{1}, \dots, \vect{q}_{2m-2})$ 
is a Gelfand-Tsetlin base of the representation $V_{\lambda}$ of
$SO(2m-1)$. 
Then $Q$ is contained in the irreducible representation of $SO(2m-2)$
whose highest weight is $\vect{q}_{2m-3}$. 
By \eqref{eq:HC map image of Pfaffian}, the image of Harish-Chandra
map of $(\PF_{2m-2})^{\mathrm{opp}}$ is 
\[
(\I)^{m-1} 
\gamma_{\lie{u}}((\PF_{2m-2})^{\mathrm{opp}}) 
= 
(-)^{m-1} T_{m-1} \cdots T_{1}. 
\]
It follows that $(\PF_{2m-2})^{\mathrm{opp}}$ acts on $Q$ by the
scalar 
\begin{align*} 
(\I)^{m-1} &(q_{2m-3,1}+m-2)\cdots(q_{2m-3,m-2}+1) q_{2m-3,m-1} %Q 
\\
&
= 
(\I)^{m-1} 
\left(\prod_{i=1}^{m-1} l_{2m-3,i}\right). % Q. 
\end{align*}
Since 
\[
a_{2m-2,0}(Q) 
= 
\I 
\frac{\prod_{i=1}^{m-1} l_{2m-3,i} \prod_{i=1}^{m} l_{2m-1,i}}
{\prod_{i=1}^{m-1} l_{2m-2,k}(l_{2m-2,k}-1)}, 
\]
there exists a constant $d_{\lambda}$, 
which depends on $\lambda$ but not on $Q \in GT(\lambda)$, 
such that 
\[- d_{\lambda} (\I)^{m} \left(\prod_{i=1}^{m-1} l_{2m-3,i}\right) Q 
= a_{2m-2,0}(Q) Q 
= \varpi_{0} Q. 
\]
We have proved the following proposition: 
\begin{proposition}
For 
$\phi_{\tau_{\lambda}}(g) 
= 
\sum_{Q \in GT(\lambda)} c(Q;g) Q 
\in C_{\tau_{\lambda}}^{\infty}(K \backslash G)$, 
the action of $\PF_{2m}$ is given by 
\begin{align}
d_{\lambda} 
L(&\PF_{2m}) \phi_{\tau_{\lambda}}(g) 
\notag\\
=&
\sum_{Q \in GT(\lambda)} 
\Bigg\{
(L(H)-m+1) c(Q;g) \varpi_{0} Q 
+ 
\sum_{i=1}^{2m-2} L(X_{i}) c(Q; g) 
[\varpi_{0}, A_{2m-1,i}] Q 
\Bigg\}
\notag\\
=& P_{0} \phi_{\tau_{\lambda}}(g). 
\notag
\end{align}

\end{proposition}

\end{document}